\documentclass[11pt,letterpaper,reqno]{amsart}
\usepackage{amssymb}
\usepackage{amsmath}
\usepackage{amsthm}
\usepackage{amsfonts}
\usepackage{enumitem}
\usepackage{booktabs}
\usepackage{graphicx}
\usepackage{xcolor}
\usepackage[T1]{fontenc}
\usepackage{doi}
\usepackage{comment}
\usepackage{hyperref}
\hypersetup{pdfstartview={FitH}}
\usepackage{bookmark}
\usepackage[capitalize,noabbrev]{cleveref}

\newtheorem{theorem}{Theorem}[section]
\newtheorem{lemma}[theorem]{Lemma}

\newtheorem{problem}[theorem]{Problem}

\theoremstyle{definition}

\newtheorem{remark}[theorem]{Remark}

\crefname{theoremletter}{Theorem}{Theorems}
\Crefname{theoremletter}{Theorem}{Theorems}

\newcommand{\Prob}{\mathbb P}
\newcommand{\E}{\mathbb E}
\newcommand{\Var}{\operatorname{Var}}
\newcommand{\Bin}{\operatorname{Bin}}
\newcommand{\Ber}{\operatorname{Ber}}
\newcommand{\supp}{\operatorname{supp}}

\begin{document}

\title{Irregular subgraph in a regular graph}

\author[T.~Cao]{Tianyue Cao}
\address{School of Mathematical Sciences, Fudan University, Shanghai 200433, P. R. China}
\email{cao\_tianyue@hotmail.com}

\author[Q.~Tang]{Quanyu Tang}
\address{School of Mathematics and Statistics, Xi'an Jiaotong University, Xi'an 710049, P. R. China}
\email{tangquanyu827@gmail.com}

\author[H.~Wu]{Hehui Wu}
\address{Shanghai Center for Mathematical Sciences, Fudan University, Shanghai 200433, P. R. China}
\email{hhwu@fudan.edu.cn}



\begin{abstract}
A conjecture of Alon and Wei states that, for any $d$-regular graph $G$ with $n$ vertices, there exists a spanning subgraph $H$ such that for all $0\le i\le d$, we have $m(H, i)$, the number of vertices in $H$ with degree $i$, is between $\frac{n}{d+1}-2$ and $\frac{n}{d+1}+2$.  We prove the conjecture for all fixed $d$ when $n$ is sufficiently large. 
More precisely, if
$q=(q_0,\ldots,q_d)$ satisfies
\[
        \sum_{i=0}^d q_i=n,\qquad
        \sum_{i=0}^d i q_i\equiv 0\pmod 2,\qquad
        \left|q_i-\frac{n}{d+1}\right|\le 1
        \quad (0\le i\le d),
\]
then there is a spanning subgraph $H\subseteq G$ such that
\[
        m(H,i)=q_i
        \qquad (0\le i\le d).
\]
\end{abstract}

\maketitle

\section{Introduction}

Let $G$ be a finite simple graph.  For a spanning subgraph $H\subseteq G$ and an
integer $i\ge 0$, write
\[
    m(H,i)=|\{v\in V(G):d_H(v)=i\}|.
\]
The following problem is a conjecture of Alon and Wei
\cite{AlonWei}.

\begin{problem}[Alon--Wei]
\label{prob:alon-wei}
Let $G$ be a finite simple $d$-regular graph on $n$ vertices.  Does there always
exist a spanning subgraph $H\subseteq G$ such that
\[
    \left|m(H,i)-\frac{n}{d+1}\right|\le 2
    \qquad (0\le i\le d)?
\]
\end{problem}

An asymptotic version of Problem~\ref{prob:alon-wei} was first obtained,
in connection with irregularity strength, by Frieze, Gould, Karo\'nski and
Pfender~\cite{FGKP}.  The random threshold construction
\[
    uv\in E(H)
    \quad\Longleftrightarrow\quad
    x(u)+x(v)\ge 1,
\]
where the labels $x(v)$ are independent uniform random variables on $[0,1]$,
already appears in their work.  Alon and Wei proved several further asymptotic
relaxations of Problem~\ref{prob:alon-wei} using this model~\cite{AlonWei}.
Fox, Luo and Pham then proved asymptotic degree balancing in the range
$d=o(n/(\log n)^{12})$~\cite{FoxLuoPham}.  Ma and Xie later introduced a
local-adjustment framework and proved the first bound independent of $n$,
namely an additive $2d^2$ bound; they also proved the cubic case of
Problem~\ref{prob:alon-wei} in a strong form~\cite{MaXie}.  The cubic case was
also resolved independently by Lu\v{z}ar, Przyby\l{}o and Sot\'ak
\cite{LuzarPrzybyloSotak}.

A very recent preprint of Montgomery, Pokrovskiy and Sudakov~\cite{MPS26} proved the asymptotic form of the Alon--Wei conjecture
in the full range $d=o(n)$.  More precisely, they showed that every
$d$-regular graph on $n$ vertices, with $d=o(n)$, has a spanning subgraph $H$
such that
\[
    m(H,i)=(1+o(1))\frac{n}{d+1}
    \qquad (0\le i\le d).
\]
This settles the asymptotic version of Problem~\ref{prob:alon-wei}.

The purpose of this paper is to prove a different, exact, polynomial-range
version of the Alon--Wei problem. Our result does not merely ask that all degree classes have size
$(1+o(1))n/(d+1)$.  Instead, subject only to the unavoidable total-size and
parity constraints, we prescribe the entire degree-count vector exactly.  Namely,
for every admissible vector $q=(q_0,\ldots,q_d)$ satisfying
\[
    \sum_i q_i=n,\qquad
    \sum_i i q_i\equiv 0\pmod 2,\qquad
    \left|q_i-\frac{n}{d+1}\right|\le 1,
\]
we construct a spanning subgraph $H\subseteq G$ with
\[
    m(H,i)=q_i
    \qquad (0\le i\le d).
\]
This exact conclusion is proved here in the polynomial range
$d\le n^{1/12-\varepsilon}$.

The argument uses the same threshold random subgraph as the initial object, but
the correction step is different.  It uses only single-edge moves, and the main
point is that it is enough to use edge types lying on the two anti-diagonals
\[
    a+b=d-2,
    \qquad
    a+b=d-1.
\]
The lattice cost of these moves is controlled not by a coarse
$\ell^1$-norm of the degree-count error, but by a twofold prefix norm.  In the
threshold model this prefix norm has expectation $O(d^2\sqrt n)$.  Combining
this with a uniform edge-atom estimate of order $d^{-4}$ gives the range
$d\le n^{1/12-\varepsilon}$ for every fixed $\varepsilon>0$.  Although
Problem~\ref{prob:alon-wei} asks for additive error $2$, the argument below
gives additive error $1$ throughout this polynomial range, after the unavoidable
parity constraint is built into the target vector.

\section{Main theorem}
\label{sec:statement}

For $0\le i\le d$, let $e_i$ denote the $i$th standard basis vector of
$\mathbb Z^{d+1}$.  Put
\begin{equation*}
        L_d=
        \left\{
        z\in\mathbb Z^{d+1}:
        \sum_{i=0}^d z_i=0,
        \quad
        \sum_{i=0}^d i z_i\equiv 0\pmod 2
        \right\}.
\end{equation*}

\begin{theorem}
\label{thm:main}
For every $\varepsilon>0$ there exists $n_0=n_0(\varepsilon)$ such that the
following holds.  Let $G$ be a simple $d$-regular graph on $n\ge n_0$ vertices,
where
\begin{equation}
\label{eq:d-range}
        2\le d\le n^{1/12-\varepsilon}.
\end{equation}
Let $q=(q_0,\ldots,q_d)$ be a nonnegative integer vector satisfying
\begin{equation}
\label{eq:target-vector}
        \sum_{i=0}^d q_i=n,
        \qquad
        \sum_{i=0}^d i q_i\equiv 0\pmod 2,
        \qquad
        \left|q_i-\frac{n}{d+1}\right|\le 1
        \quad (0\le i\le d).
\end{equation}
Then there exists a spanning subgraph $H\subseteq G$ such that
\begin{equation}
\label{eq:exact-target}
        m(H,i)=q_i
        \qquad (0\le i\le d).
\end{equation}
Consequently, every such $G$ has a spanning subgraph $H\subseteq G$ satisfying
\[
        \left|m(H,i)-\frac{n}{d+1}\right|\le 1
        \qquad (0\le i\le d).
\]
\end{theorem}

\begin{remark}
\label{rem:mps-comparison}
Theorem~\ref{thm:main} was obtained before the public appearance of the
preprint of Montgomery, Pokrovskiy and Sudakov~\cite{MPS26}.  Their theorem and
Theorem~\ref{thm:main} are complementary rather than comparable.  The former
treats the asymptotic problem in the full range $d=o(n)$, giving degree-class
sizes $(1+o(1))n/(d+1)$.  The latter works only in the polynomial range
$d\le n^{1/12-\varepsilon}$, but realizes every admissible balanced
degree-count vector exactly, subject only to the total-size and parity
constraints in~\eqref{eq:target-vector}.
\end{remark}

\subsection{Proof overview}
Before entering the technical estimates, we describe the mechanism of the proof.
The proof uses the probabilistic method only to find a good starting subgraph;
after such a starting subgraph has been found, the rest of the argument is
deterministic.

Choose independent random variables $X_v\sim {\rm Unif}[0,1]$, one for each
vertex $v\in V(G)$, and define the threshold subgraph $H_0\subseteq G$ by
\[
    uv\in E(H_0)
    \quad\Longleftrightarrow\quad
    X_u+X_v\ge 1 .
\]
Let
\[
    M_i=m(H_0,i),\qquad M=(M_0,\ldots,M_d).
\]
The target vector $q$ in Theorem~\ref{thm:main} is fixed in advance, and the
discrepancy to be corrected is
\[
    z=q-M.
\]
The choice of $H_0$ is useful because, for every vertex $v$, the degree of $v$
in $H_0$ is uniformly distributed on $\{0,\ldots,d\}$. Thus, $M$ is already
centered around the desired balanced vector.

The proof then has two separate tasks.

The first task is algebraic.  Suppose that an absent edge whose endpoints
currently have degrees $a$ and $b$ is added.  Then the degree-count vector
changes by
\[
    \beta_{a,b}=-e_a-e_b+e_{a+1}+e_{b+1}.
\]
Deleting a present edge whose endpoints with degree $a$ and $b$ after deleting gives the opposite change.  Therefore, in order to
correct the error $z=q-M$, we want to write $z$ as an integer combination of
such vectors $\beta_{a,b}$.

The proof does not use all possible pairs $(a,b)$.  It uses only the two
anti-diagonals
\[
    a+b=d-2
    \qquad\text{and}\qquad
    a+b=d-1.
\]
The reason for this restriction is probabilistic: for these pairs, the natural
label point
\[
    \left(\frac{a}{d-1},\frac{b}{d-1}\right)
\]
lies close to the threshold line $x+y=1$. Hence one can find small rectangles
near this point on either side of the line $x+y=1$, giving both non-edges that
can be added and edges that can be deleted.  This is the bridge between the
lattice part and the random edge-supply part.

The refined lattice lemma, Lemma~\ref{lem:refined-lattice}, proves that every
valid discrepancy vector $z\in L_d$ has a representation
\[
    z=
    \sum_{\substack{0\le a,b\le d-1\\ a+b\in\{d-2,d-1\}}}
    \lambda_{a,b}\beta_{a,b},
\]
with controlled cost
\[
    \sum_{a,b}|\lambda_{a,b}|
    \le
    C_L\left(\sum_{j=1}^{d-1}|P_j(z)|+d|S(z)|\right).
\]
Thus the lattice lemma produces a ``shopping list'': if
$\lambda_{a,b}>0$, we need to add that many edges of type $(a,b)$; if
$\lambda_{a,b}<0$, we need to delete that many edges of type $(a,b)$.

The second task is probabilistic.  For every allowed signed type, we must show
that the random threshold graph supplies many candidate edges of that type.
The key local estimate is Lemma~\ref{lem:edge-atom}.  It says that, for every
fixed edge $e\in E(G)$ and every allowed signed type,
\[
    \Prob(e\text{ is a candidate of that signed type})
    \ge c_{\rm at}d^{-4}.
\]
Since $G$ has $nd/2$ edges, the expected number of candidates of one fixed
signed type is therefore of order
\[
    nd\cdot d^{-4}=nd^{-3}.
\]
Lemma~\ref{lem:supply} turns this local estimate into a simultaneous supply
statement: with probability tending to one, every signed anti-diagonal type has
at least
\[
    c_Snd^{-3}
\]
candidate edges.

It remains to compare supply with demand.  Lemma~\ref{lem:expected-refined}
shows that the expected refined size of the random discrepancy is at most
\[
    C(d^2\sqrt n+d^3).
\]
In the range $d\le n^{1/12-\varepsilon}$, this is $O(d^2\sqrt n)$.  Hence, with
positive probability, the same threshold graph $H_0$ simultaneously satisfies

\[
    \text{every signed type has supply at least } c_Snd^{-3},
\]
and
\[
    \text{the shopping list has total size at most } C d^2\sqrt n.
\]
We then fix such a realization of $H_0$.

The final step is deterministic.  We choose the required candidate edges
greedily, insisting that all selected edges are vertex-disjoint.  If
\[
    R\le C d^2\sqrt n
\]
edges have to be selected in total, then previously selected edges can make at
most $O(dR)=O(d^3\sqrt n)$ candidates of any fixed signed type unavailable.
But each signed type has supply at least $c_Snd^{-3}$, and
\[
    nd^{-3}\gg d^3\sqrt n
\]
is exactly the inequality guaranteed by
\[
    d\le n^{1/12-\varepsilon}.
\]
Thus the greedy procedure can choose all required candidates vertex-disjointly.

After all choices are made, we add the selected positive candidates and delete
the selected negative candidates.  Because the selected edges are
vertex-disjoint, their degree-count effects do not interfere.  The total change
in the degree-count vector is exactly $z=q-M$, so the resulting subgraph has
degree-count vector $q$.


\subsection{Organization of the paper}
Section~\ref{sec:local-prob} proves the
local probability estimates, including the edge-atom lower bound.
Section~\ref{sec:lattice} proves the anti-diagonal lattice lemma.
Section~\ref{sec:random-discrepancy} bounds the refined discrepancy of the
threshold graph.  Section~\ref{sec:completion} combines these ingredients with
the simultaneous supply estimate and the greedy vertex-disjoint selection.

\section{Local probabilistic estimates}
\label{sec:local-prob}

\subsection{One-dimensional estimates}
\label{subsec:one-dimensional}

The estimates in this subsection are used only in the proof of the
edge-atom estimate, Lemma~\ref{lem:edge-atom}.  Their purpose is to give
uniform lower bounds for exact point probabilities of binomial-type random
variables, including boundary cases where the success probability is as small
as order $1/d$.  Readers interested first in the global structure of the proof
may skip this subsection on a first reading and return to it when reading
Lemma~\ref{lem:edge-atom}.

We begin with two standard tools for sums of independent Bernoulli random
variables.  The first is a uniform lattice local approximation for
Poisson-binomial variables; the second is Le Cam's total-variation form of the
Poisson approximation. We quote these estimates in precisely the form needed
below. Here $\mathcal L(X)$ denotes the law of $X$, and
$d_{\mathrm{TV}}(\mu,\nu)=\sup_A|\mu(A)-\nu(A)|$.

We shall use the following Bernoulli local approximation and Poisson-coupling estimates.
\begin{lemma}
\label{lem:bernoulli-tools}
Let $\xi_1,\ldots,\xi_m$ be independent Bernoulli random variables with
parameters $p_1,\ldots,p_m$, and put
\[
        T=\sum_{i=1}^m \xi_i,
        \qquad
        \mu=\E T,
        \qquad
        \sigma^2=\Var(T)=\sum_{i=1}^m p_i(1-p_i).
\]
There is an absolute constant $C$ such that, whenever $\sigma\ge1$,
\begin{equation}
\label{eq:bernoulli-local-approximation}
        \sup_{k\in\mathbb Z}
        \left|
        \Prob(T=k)
        -\frac{1}{\sqrt{2\pi}\,\sigma}
        \exp\!\left(-\frac{(k-\mu)^2}{2\sigma^2}\right)
        \right|
        \le \frac{C}{\sigma^2}.
\end{equation}

Moreover, let $\zeta_i$, $1\le i\le m$, be independent Bernoulli random
variables with respective parameters $r_i$.  Put
\[
        Z=\sum_{i=1}^m\zeta_i,
        \qquad
        \lambda=\sum_{i=1}^m r_i.
\]
Then
\begin{equation}
\label{eq:le-cam-coupling}
        d_{\mathrm{TV}}\!\left(
        \mathcal L(Z),\operatorname{Poisson}(\lambda)
        \right)
        \le \sum_{i=1}^m r_i^2.
\end{equation}
\end{lemma}

\begin{proof}
We first prove \eqref{eq:bernoulli-local-approximation}.  Deterministic
summands may be deleted: a summand with parameter $0$ has no effect, while a
summand with parameter $1$ shifts both $T$ and $\mu$ by one and leaves
$\sigma^2$ unchanged.  Thus, after an integer shift, we may assume that all
remaining parameters lie in $(0,1)$.

For $k\in\{0,\ldots,m\}$, \cite[Theorem~1.2]{AuldNeammanee} gives
\[
        \left|
        \Prob(T=k)
        -\frac{1}{\sqrt{2\pi}\,\sigma}
        \exp\!\left(-\frac{(k-\mu)^2}{2\sigma^2}\right)
        \right|
        \le
        \frac{3.23}{\sigma^2}
        +\frac{1.35}{\sigma^3}
        +\frac{0.25}{\sigma^4}.
\]
Since $\sigma\ge1$, this is at most $C\sigma^{-2}$.

It remains to consider integers outside the support.  After the above deletion
of deterministic summands, the support is contained in $\{0,\ldots,m\}$, and
\[
        \sigma^2=\sum_i p_i(1-p_i)\le \sum_i p_i=\mu,
        \qquad
        \sigma^2\le \sum_i(1-p_i)=m-\mu .
\]
If $k<0$, then $|k-\mu|\ge \mu\ge\sigma^2$.  If $k>m$, then, since
$k$ is an integer, $|k-\mu|\ge m+1-\mu\ge m-\mu\ge\sigma^2$.  Therefore
\[
        \frac{1}{\sqrt{2\pi}\,\sigma}
        \exp\!\left(-\frac{(k-\mu)^2}{2\sigma^2}\right)
        \le
        \frac{1}{\sqrt{2\pi}\,\sigma}e^{-\sigma^2/2}
        \le
        C\sigma^{-2},
        \qquad \sigma\ge1.
\]
This proves \eqref{eq:bernoulli-local-approximation}.

We next prove \eqref{eq:le-cam-coupling}. This follows from \cite[Proposition~1]{LeCam} after translating his norm
convention. Indeed,
Le Cam uses
\[
        \|\mu\|=\sup_{|f|\le1}\left|\int f\,d\mu\right|,
\]
which equals twice our convention
$d_{\mathrm{TV}}(\mu,\nu)=\sup_A|\mu(A)-\nu(A)|$ for differences of
probability measures. \cite[Proposition~1]{LeCam} gives, for a sum of independent
Bernoulli variables,
\[
        \|\mathcal L(Z)-\operatorname{Poisson}(\lambda)\|
        \le 2\sum_i r_i^2 .
\]
Thus \eqref{eq:le-cam-coupling} follows.  We also recall the following
short coupling proof, which fixes the normalization.

Let $P_r$ denote a Poisson random variable of mean $r$.  For $0\le r\le1$,
a direct calculation gives
\[
        d_{\mathrm{TV}}\bigl(\Ber(r),\mathcal L(P_r)\bigr)
        =r(1-e^{-r})\le r^2 .
\]
Indeed, the only positive part of the signed measure
$\Ber(r)-\mathcal L(P_r)$ is at the point $1$.

For each $i$, choose a maximal coupling
$(\widetilde\zeta_i,P_i)$ of $\Ber(r_i)$ and $\operatorname{Poisson}(r_i)$,
and take these pairs independent over $i$.  Then
$\sum_i\widetilde\zeta_i$ has the same law as $Z$, while
$\sum_iP_i$ has the Poisson distribution of mean $\lambda$.  Hence, by the
coupling inequality,
\[
        d_{\mathrm{TV}}\!\left(
        \mathcal L(Z),\operatorname{Poisson}(\lambda)
        \right)
        \le
        \Prob\!\left(\sum_i\widetilde\zeta_i\ne\sum_iP_i\right)
        \le
        \sum_i\Prob(\widetilde\zeta_i\ne P_i)
        \le
        \sum_i r_i^2 .
\]
This proves \eqref{eq:le-cam-coupling}.
\end{proof}

We next state the one-dimensional estimates in the precise form needed for the
edge-atom argument.  The extra geometric hypotheses in part~2 are essential in
the bounded-variance boundary regimes.

\begin{lemma}
\label{lem:local-estimates}
Fix constants $\delta>0$, $C_0>0$ and $K>0$.  There are constants
$c_B,c_T>0$, depending only on $\delta,C_0,K$, such that the following hold for
all integers $s\ge 1$.

\begin{enumerate}
\item Let $B\sim\Bin(s,p)$, where
\[
        \frac{\delta}{s}\le p\le 1-\frac{\delta}{s}.
\]
If $k\in\{0,1,\ldots,s\}$ and $|k-sp|\le K$, then
\begin{equation}
\label{eq:bin-local}
        \Prob(B=k)\ge c_B(s+1)^{-1/2}.
\end{equation}

\item Let $0\le a\le b\le s$ and suppose
\[
        a+b\in\{s-1,s\}.
\]
Let $x,y\in[0,1]$ satisfy
\begin{equation}
\label{eq:threshold-local-geometry}
        x\le y,
        \qquad
        |sx-a|+|sy-b|\le C_0,
        \qquad
        x,\ y-x,\ 1-y\ge \frac{\delta}{s}.
\end{equation}
Let $0\le t\le s$ be an integer, and let $c_*$ be an integer such that
\begin{equation*}
        0\le c_*\le \min\{a,t\},
        \qquad
        \left|c_*-\frac{at}{s}\right|\le K.
\end{equation*}
Put
\[
         p_{xy}=\frac{y-x}{1-x},
\]
and let
\[
        W=c_*+\Bin(t-c_*,p_{xy})+\Bin(s-t,y),
\]
where the two binomial random variables are independent.  Then
\begin{equation}
\label{eq:threshold-conv-local}
        \Prob(W=b)\ge c_T(s+1)^{-1/2}.
\end{equation}

\item Let $0\le a,t\le s$ be integers. Let $Y$ have the hypergeometric
distribution obtained by drawing $a$ points from a set of size $s$, of which
$t$ are marked. Then there is an integer
$c_*\in\supp(Y)$ such that
\begin{equation}
\label{eq:hyp-local}
        \Prob(Y=c_*)\ge \frac{1}{s+1},
        \qquad
        \left|c_*-\frac{at}{s}\right|\le 1.
\end{equation}
\end{enumerate}
\end{lemma}

\begin{proof}
We first prove the binomial estimate.  Suppose, to the contrary, that there is
a sequence of admissible parameters for which
\[
        \sqrt{s+1}\,\Prob(B=k)\longrightarrow0.
\]
For bounded $s$, compactness of the admissible interval for $p$ and positivity
of every binomial mass on its support give a uniform positive lower bound, so
necessarily $s\to\infty$.  Put $\sigma^2=sp(1-p)$.  After passing to a
subsequence, either $\sigma^2\to\infty$ or $\sigma^2=O(1)$.

In the first case, Lemma~\ref{lem:bernoulli-tools} and $|k-sp|\le K$ give
\[
        \Prob(B=k)
        \ge
        \frac{1}{\sqrt{2\pi}\,\sigma}
        \exp\!\left(-\frac{K^2}{2\sigma^2}\right)
        -\frac{C}{\sigma^2}
        \ge \frac{c}{\sigma}
        \ge c s^{-1/2}
\]
for all sufficiently large $s$, a contradiction.

Suppose now that $\sigma^2=O(1)$.  Passing to a further subsequence, assume
first that $p\le1/2$.  Then
\[
        \delta\le sp\le 2\sigma^2=O(1).
\]
Thus $k$ ranges over a fixed finite set.  Passing to another subsequence, we may
assume that $k$ is fixed and $sp\to\lambda\in(0,\infty)$.  Since $p\to0$,
\eqref{eq:le-cam-coupling} gives
\[
        d_{\mathrm{TV}}\bigl(\mathcal L(B),
        \operatorname{Poisson}(sp)\bigr)
        \le sp^2\longrightarrow0.
\]
It follows that
\[
        \Prob(B=k)\longrightarrow e^{-\lambda}\frac{\lambda^k}{k!}>0,
\]
again a contradiction.  If $p\ge1/2$, apply the same argument to $s-B$, whose
success parameter is $1-p$ and whose target is $s-k$.  This proves
\eqref{eq:bin-local}.

We turn to the threshold-convolution estimate.  The hypotheses imply
$0<p_{xy}<1$.  Since also $0<y<1$, the two binomial variables in $W$ have full
interval support, and therefore
\[
        \supp(W)=\{c_*,c_*+1,\ldots,s\}.
\]
As $c_*\le a\le b\le s$, the target $b$ lies in this support.  Put
\[
        \bar c=\frac{at}{s}.
\]
Using $1-p_{xy}=(1-y)/(1-x)$ and
$p_{xy}-y=-x(1-y)/(1-x)$, we obtain
\begin{align*}
        p_{xy}t+(1-p_{xy})\bar c+(s-t)y-sy
        &=
        \frac{t(1-y)}{1-x}\left(\frac{a}{s}-x\right).
\end{align*}
The right-hand side is $O(1)$ by
\eqref{eq:threshold-local-geometry}.  Replacing $\bar c$ by $c_*$ changes the
left-hand side by at most $K$, and $|sy-b|\le C_0$.  Hence
\begin{equation}
\label{eq:threshold-local-target-near-mean}
        |\E W-b|\le C,
\end{equation}
where $C$ depends only on $C_0$ and $K$.

Suppose that \eqref{eq:threshold-conv-local} fails.  We may then choose a
sequence of admissible parameters such that
\begin{equation}
\label{eq:threshold-countersequence}
        \sqrt{s+1}\,\Prob(W=b)\longrightarrow0.
\end{equation}
For bounded $s$, there are only finitely many choices of the integer parameters.
For each such choice the admissible set of $(x,y)$ is compact,
$p_{xy},y\in(0,1)$, and $b$ lies in the support, so the relevant point
probability has a positive minimum. Therefore $s\to\infty$ along the countersequence.

Let
\[
        \sigma_s^2=
        (t-c_*)p_{xy}(1-p_{xy})+(s-t)y(1-y)
\]
be the variance of $W$. If $\sigma_s^2$ is unbounded, pass to a subsequence on
which $\sigma_s^2\to\infty$.  Lemma~\ref{lem:bernoulli-tools} and
\eqref{eq:threshold-local-target-near-mean} give
\[
        \Prob(W=b)
        \ge
        \frac{1}{\sqrt{2\pi}\,\sigma_s}
        \exp\!\left(-\frac{C^2}{2\sigma_s^2}\right)
        -\frac{C'}{\sigma_s^2}
        \ge \frac{c}{\sigma_s}
        \ge c s^{-1/2},
\]
because $\sigma_s^2\le s/4$.  This contradicts
\eqref{eq:threshold-countersequence}.  We may therefore pass to a subsequence
on which
\begin{equation}
\label{eq:threshold-bounded-var}
        \sigma_s^2=O(1).
\end{equation}

Pass to a further subsequence so that $a/s\to\alpha$ and $t/s\to\tau$.  Since
$a+b\in\{s-1,s\}$ and \eqref{eq:threshold-local-geometry} holds,
\[
        x\to\alpha,
        \qquad
        y\to1-\alpha,
        \qquad
        \frac{c_*}{s}\to\alpha\tau,
        \qquad
        0\le\alpha\le\frac12.
\]
Moreover
\[
        p_{xy}\to p_0:=\frac{1-2\alpha}{1-\alpha}.
\]
Dividing the variance by $s$ gives the limiting variance density
\[
        \tau(1-\alpha)p_0(1-p_0)
        +(1-\tau)\alpha(1-\alpha).
\]
By \eqref{eq:threshold-bounded-var}, this expression is zero.  Hence either
\begin{equation}
\label{eq:threshold-boundary-one}
        \alpha=0,
\end{equation}
or
\begin{equation}
\label{eq:threshold-boundary-two}
        \alpha=\frac12
        \quad\text{and}\quad
        \tau=1.
\end{equation}

First consider \eqref{eq:threshold-boundary-one}.  Put
\[
        \eta_s=s-(a+b)\in\{0,1\}.
\]
Then
\begin{equation}
\label{eq:alpha-zero-s-one-minus-y}
        s(1-y)=s-b+(b-sy)=a+\eta_s+O(1).
\end{equation}
Count failures by writing
\[
        F:=s-W
        =\Bin(t-c_*,1-p_{xy})+\Bin(s-t,1-y),
\]
and let
\[
        \lambda_s=
        (t-c_*)(1-p_{xy})+(s-t)(1-y)=\E F.
\]
Since $p_{xy},y\to1$, for all sufficiently large $s$,
\[
        \sigma_s^2\ge\frac12\lambda_s.
\]
Thus \eqref{eq:threshold-bounded-var} implies $\lambda_s=O(1)$.  Also
\[
        \lambda_s
        =(1-y)\left(\frac{t-c_*}{1-x}+s-t\right)
        \ge(1-y)(s-c_*).
\]
Since $c_*/s\to0$, it follows that $1-y=O(1/s)$.  By
\eqref{eq:alpha-zero-s-one-minus-y},
\begin{equation*}
        a=O(1),
        \qquad
        s-b=a+\eta_s=O(1),
        \qquad
        c_*=\frac{at}{s}+O(1)=O(1).
\end{equation*}
On the other hand, $1-y\ge\delta/s$ and $s-c_*\ge s/2$ for all large $s$, so
$\lambda_s\ge\delta/2$.  Pass to a further subsequence on which
$s-b=k$ is fixed and $\lambda_s\to\lambda\in(0,\infty)$. The largest
Bernoulli parameter occurring in $F$ tends to zero, because
\[
        1-y=O(1/s),
        \qquad
        1-p_{xy}=\frac{1-y}{1-x}=O(1/s).
\]
Consequently, \eqref{eq:le-cam-coupling} gives
\[
        d_{\mathrm{TV}}\bigl(\mathcal L(F),
        \operatorname{Poisson}(\lambda_s)\bigr)
        \le
        \lambda_s\max\{1-p_{xy},1-y\}
        \longrightarrow0.
\]
Therefore
\[
        \Prob(W=b)=\Prob(F=k)
        \longrightarrow e^{-\lambda}\frac{\lambda^k}{k!}>0,
\]
contradicting \eqref{eq:threshold-countersequence}.

It remains to consider \eqref{eq:threshold-boundary-two}.  Put
\[
        \Delta_s=b-a\ge0,
        \qquad
        \eta_s=s-(a+b)\in\{0,1\}.
\]
Then
\begin{equation*}
        s(y-x)=(sy-b)+(b-a)+(a-sx)=\Delta_s+O(1).
\end{equation*}
Since $x\to1/2$, the quantities $p_{xy}$ and $y-x$ are comparable.  Also
$t-c_*\sim s/2$ and $p_{xy}\to0$.  Hence the first term in the variance is
comparable to $s(y-x)$, and \eqref{eq:threshold-bounded-var} implies
\begin{equation*}
        \Delta_s=O(1),
        \qquad
        p_{xy}=O(1/s).
\end{equation*}
The second term in the variance, together with $y\to1/2$, gives
\begin{equation*}
        s-t=O(1).
\end{equation*}
Furthermore,
\[
        b-c_*
        =(b-a)+a\left(1-\frac{t}{s}\right)
        +\left(\frac{at}{s}-c_*\right)=O(1).
\]
Set
\[
        n_s=t-c_*,
        \qquad
        m_s=s-t.
\]
The lower bound $y-x\ge\delta/s$ and the relation $n_s\sim s/2$ imply
$n_s p_{xy}\ge c>0$, while bounded variance and $p_{xy}\to0$ imply
$n_s p_{xy}=O(1)$.
After passing to a further subsequence, we may assume that
\[
        b-c_*=k,
        \qquad
        m_s=m,
        \qquad
        n_s p_{xy}\to\lambda\in(0,\infty),
\]
where $k,m$ are fixed nonnegative integers.  By
\eqref{eq:le-cam-coupling},
\[
        d_{\mathrm{TV}}\bigl(
        \mathcal L(\Bin(n_s,p_{xy})),
        \operatorname{Poisson}(n_s p_{xy})
        \bigr)
        \le n_s p_{xy}^2=(n_s p_{xy})p_{xy}\longrightarrow0.
\]
Also $\Bin(m_s,y)$ converges in distribution, equivalently pointwise on its
fixed finite support, to $\Bin(m,1/2)$.  Therefore
\[
\begin{aligned}
        \Prob(W=b)
        &\longrightarrow
        \sum_{j=0}^{\min\{m,k\}}
        \binom{m}{j}2^{-m}
        e^{-\lambda}\frac{\lambda^{k-j}}{(k-j)!}                 \\
        &\ge
        2^{-m}e^{-\lambda}\frac{\lambda^k}{k!}>0,
\end{aligned}
\]
again contradicting \eqref{eq:threshold-countersequence}.  This proves
\eqref{eq:threshold-conv-local}.

Finally, consider the hypergeometric variable $Y$.  For two consecutive points
in its support,
\[
        \frac{\Prob(Y=c+1)}{\Prob(Y=c)}
        =
        \frac{(t-c)(a-c)}{(c+1)(s-t-a+c+1)}.
\]
The ratio is at least one exactly when
\[
        c+1\le \frac{(a+1)(t+1)}{s+2}.
\]
Thus a mode may be chosen as
\[
        m=\left\lfloor\frac{(a+1)(t+1)}{s+2}\right\rfloor;
\]
when the displayed fraction is an integer, $m-1$ is also a mode.  We record why
this value indeed lies in the hypergeometric support.  The support is
\[
        \max\{0,a+t-s\}\le c\le \min\{a,t\}.
\]
The displayed fraction is nonnegative and is strictly smaller than both $a+1$
and $t+1$, while
\[
        \frac{(a+1)(t+1)}{s+2}-(a+t-s)
        =
        \frac{(s-a)(s-t)+(s-a)+(s-t)+1}{s+2}
        >0.
\]
Hence its floor belongs to the support.  Moreover,
\[
        0\le
        \frac{(a+1)(t+1)}{s+2}-\frac{at}{s}
        =
        \frac{a(s-t)+t(s-a)+s}{s(s+2)}
        <1.
\]
Indeed, for fixed $t$ the numerator is linear in $a$, and its maximum on
$0\le a\le s$ is at most $s^2+s<s(s+2)$.  Hence
$|m-at/s|\le1$.  Since the support contains at most $s+1$ points, the mass at a
mode is at least $1/(s+1)$.  This proves \eqref{eq:hyp-local}.
\end{proof}

\subsection{Candidate edges near the threshold}
\label{subsec:edge-atoms}

Let $(X_v)_{v\in V(G)}$ be independent uniform random variables on $[0,1]$ and
define the threshold random subgraph $H_0\subseteq G$ by
\begin{equation*}
        uv\in E(H_0)
        \quad\Longleftrightarrow\quad
        X_u+X_v\ge 1.
\end{equation*}
For an edge $e=uv\in E(G)$, let $D_e(u)$ and $D_e(v)$ be the degrees of $u$ and
$v$ in $H_0$ after ignoring the edge $e$ itself.

For $0\le a\le b\le d-1$, a positive candidate of type $(a,b)$ is an edge
$e=uv\notin E(H_0)$ such that
\[
        \{D_e(u),D_e(v)\}=\{a,b\}.
\]
Adding such an edge changes the degree-count vector by
\begin{equation*}
        \beta_{a,b}
        =
        -e_a-e_b+e_{a+1}+e_{b+1}.
\end{equation*}
A negative candidate of type $(a,b)$ is an edge $e=uv\in E(H_0)$ such that
\[
        \{D_e(u),D_e(v)\}=\{a,b\}.
\]
Deleting such an edge changes the degree-count vector by $-\beta_{a,b}$.

The following geometric lemma supplies rectangles on both sides of the threshold
line near the two anti-diagonals.
\begin{lemma}
\label{lem:rectangles}
There are absolute constants $\delta>0$ and $C_0>0$ with the following property.
Let $s\ge 1$, let $0\le a\le b\le s$, and suppose
\[
        a+b\in\{s-1,s\}.
\]
For each sign $\sigma\in\{-,+\}$ there is a rectangle
\[
        R=I\times J\subseteq[0,1]^2
\]
of area at least $\delta s^{-2}$ such that every $(x,y)\in R$ satisfies
\begin{equation*}
        x\le y,
        \qquad
        |sx-a|+|sy-b|\le C_0,
        \qquad
        x,\ y-x,\ 1-y\ge \frac{\delta}{s},
\end{equation*}
and
\begin{equation}
\label{eq:rect-side}
        x+y<1
        \quad\text{if }\sigma=-,
        \qquad
        x+y>1
        \quad\text{if }\sigma=+.
\end{equation}
\end{lemma}

\begin{proof}
Take
\[
        \rho=\frac1{32},
        \qquad
        \delta=\frac1{256},
        \qquad
        C_0=2.
\]
Write
\[
        x=\frac{a+\alpha}{s},
        \qquad
        y=\frac{b+\beta}{s},
        \qquad
        \eta=s-(a+b)\in\{0,1\}.
\]
Choose a center $(\alpha_0,\beta_0)$ from the following table:
\[
\begin{array}{c|c|c}
\eta & \sigma & (\alpha_0,\beta_0)\\ \hline
1 & - & (1/4,1/2)\\
1 & + & (1/2,3/4)\\
0 & -,\ a=0 & (1/4,-1/2)\\
0 & -,\ a>0 & (-1/2,1/4)\\
0 & +,\ b=s & (1/2,-1/4)\\
0 & +,\ b<s & (1/4,1/2)
\end{array}
\]
Let $\alpha$ and $\beta$ vary independently in the intervals of radius $\rho$
around $\alpha_0$ and $\beta_0$, respectively, and let $R$ be the resulting
rectangle in the $(x,y)$-plane.  Its area is
\[
        \frac{(2\rho)^2}{s^2}=\frac{1}{256s^2}=\delta s^{-2}.
\]

At every center in the table, direct substitution shows that each of
\[
        a+\alpha_0,
        \qquad
        s-b-\beta_0,
        \qquad
        b-a+\beta_0-\alpha_0
\]
is at least $1/4$.  Under the allowed perturbations, the first two quantities
decrease by at most $\rho$, and the third decreases by at most $2\rho$.
Consequently,
\[
        a+\alpha,
        \quad
        s-b-\beta,
        \quad
        b-a+\beta-\alpha
        \ge \frac{3}{16}>\delta.
\]
These inequalities imply $R\subseteq[0,1]^2$, $x\le y$, and
\[
        x,\ y-x,\ 1-y\ge\frac{\delta}{s}.
\]
Also $|\alpha|+|\beta|<2$, so
$|sx-a|+|sy-b|\le C_0$.

Finally, at every center the distance between $\alpha_0+\beta_0$ and $\eta$ on
the required side is at least $1/4$.  Perturbing both coordinates changes their
sum by at most $2\rho=1/16$.  Hence
\[
        \alpha+\beta<\eta
        \quad\text{for }\sigma=-,
        \qquad
        \alpha+\beta>\eta
        \quad\text{for }\sigma=+.
\]
Since $s(x+y-1)=\alpha+\beta-\eta$, this is exactly
\eqref{eq:rect-side}.
\end{proof}

The next lemma is the uniform edge-atom estimate.
\begin{lemma}
\label{lem:edge-atom}
There is an absolute constant $c_{\rm at}>0$ such that the following holds.  Let
$G$ be a simple $d$-regular graph with $d\ge2$, let $e=uv\in E(G)$, and let
$0\le a\le b\le d-1$ satisfy
\begin{equation*}
        a+b\in\{d-2,d-1\}.
\end{equation*}
Then
\begin{equation}
\label{eq:positive-atom}
        \Prob(e\text{ is a positive candidate of type }(a,b))
        \ge c_{\rm at}d^{-4},
\end{equation}
and
\begin{equation}
\label{eq:negative-atom}
        \Prob(e\text{ is a negative candidate of type }(a,b))
        \ge c_{\rm at}d^{-4}.
\end{equation}
\end{lemma}

\begin{proof}
Put $s=d-1$.  We prove the positive estimate first.  It is enough to
lower-bound the ordered event $D_e(u)=a$ and $D_e(v)=b$ on a region where
$X_u\le X_v$, since this event is contained in the unordered type $(a,b)$.
Apply Lemma~\ref{lem:local-estimates} throughout this proof with the constants
$\delta,C_0$ from Lemma~\ref{lem:rectangles} and with
$K=\max\{C_0,1\}$.  By Lemma~\ref{lem:rectangles}, there is a rectangle $R$
of area $\Omega(s^{-2})$ such that, for all $(x,y)\in R$,
\begin{equation}
\label{eq:edge-atom-rect}
        x\le y,
        \qquad
        |sx-a|+|sy-b|\le C_0,
        \qquad
        x,\ y-x,\ 1-y\ge \frac{\delta}{s},
        \qquad
        x+y<1.
\end{equation}
Thus, on the event $X_u=x$ and $X_v=y$ with $(x,y)\in R$, the edge $uv$ is
absent from $H_0$.

Let $t$ be the number of common neighbours of $u$ and $v$ other than each other.
Conditional on $X_u=x$ and $X_v=y$,
\[
        D_e(u)\sim\Bin(s,x).
\]
By \eqref{eq:bin-local} and \eqref{eq:edge-atom-rect},
\begin{equation}
\label{eq:u-local}
        \Prob(D_e(u)=a\mid X_u=x,X_v=y)
        \ge c s^{-1/2}.
\end{equation}

Now condition further on $D_e(u)=a$.  Among the $t$ common neighbours of $u$ and
$v$, let $Y$ be the number already joined to $u$ in $H_0$.  Conditional on
$D_e(u)=a$, the random variable $Y$ is hypergeometric with parameters
$(s,t,a)$.  By \eqref{eq:hyp-local}, there is an integer $c_*\in\supp(Y)$ such
that
\begin{equation}
\label{eq:cstar}
        \Prob(Y=c_*\mid D_e(u)=a,X_u=x,X_v=y)
        \ge \frac{1}{s+1},
        \qquad
        \left|c_*-\frac{at}{s}\right|\le 1.
\end{equation}
Moreover $c_*\le a\le s/2$, because $a\le b$ and $a+b\le s$.

Condition now on $D_e(u)=a$ and $Y=c_*$.  The $c_*$ common neighbours already
joined to $u$ are automatically joined to $v$, since $y\ge x$.  Conditional on
the identities of the remaining $t-c_*$ common neighbours, their labels are still
independent with conditional distribution uniform on $[0,1-x)$, so each is
joined to $v$ with conditional probability
\[
         p_{xy}=\frac{y-x}{1-x}.
\]
The $s-t$ private neighbours of $v$ are independent of this conditioning and are
joined to $v$ independently with probability $y$.  The preceding description does
not depend on the particular identities of the common neighbours, only on their
number $c_*$.  Hence, after removing the conditioning on identities,
\begin{equation*}
        D_e(v)
        \stackrel{d}{=}
        c_*+\Bin(t-c_*,p_{xy})+\Bin(s-t,y),
\end{equation*}
where the two binomial random variables on the right-hand side are independent.
The hypotheses of the threshold-convolution estimate are then exactly
\eqref{eq:edge-atom-rect} and \eqref{eq:cstar}. Therefore
\begin{equation}
\label{eq:v-local}
        \Prob(D_e(v)=b\mid D_e(u)=a,Y=c_*,X_u=x,X_v=y)
        \ge c s^{-1/2}.
\end{equation}

Multiplying \eqref{eq:u-local}, \eqref{eq:cstar}, and \eqref{eq:v-local}, and
then integrating over the rectangle $R$, gives
\[
        \Prob(e\text{ is a positive candidate of type }(a,b))
        \ge c s^{-2}\cdot s^{-1/2}\cdot s^{-1}\cdot s^{-1/2}
        \ge c_{\rm at}d^{-4}.
\]
This proves \eqref{eq:positive-atom}.

For the negative estimate, use the rectangle supplied by Lemma~\ref{lem:rectangles}
with $x+y>1$.  Then $uv\in E(H_0)$.  The variables $D_e(u)$ and $D_e(v)$ ignore
the contribution of $uv$ itself, so the conditional distribution of all other
incident edges is the same as in the argument above.  The same multiplication
therefore proves \eqref{eq:negative-atom}.
\end{proof}

\section{Anti-diagonal lattice reduction}
\label{sec:lattice}

For $0\le i\le d-1$, write
\[
        f_i=e_{i+1}-e_i.
\]
Then
\[
        \beta_{a,b}=f_a+f_b.
\]
In the lattice lemma below we use this formula for ordered pairs $(a,b)$; of
course $\beta_{a,b}=\beta_{b,a}$.
For $z\in L_d$, define integers $c_0,\ldots,c_{d-1}$ by
\begin{equation*}
        z=\sum_{i=0}^{d-1} c_i f_i,
        \qquad
        c_i=-\sum_{r=0}^i z_r.
\end{equation*}
Also define
\begin{equation}
\label{eq:S-P-def}
        S(z)=\sum_{i=0}^{d-1}c_i
        =
        \sum_{i=0}^d i z_i,
        \qquad
        P_j(z)=\sum_{i=j}^{d-1}c_i
        \quad (1\le j\le d-1).
\end{equation}
The equality for $S(z)$ follows from $\sum_i z_i=0$.  Since $z\in L_d$, $S(z)$
is even.

For later use, extend $S$ and $P_j$ linearly to arbitrary
$w=(w_0,\ldots,w_d)\in\mathbb R^{d+1}$.  For $1\le j\le d-1$, put
\[
        \phi_j(k)=
        \begin{cases}
        d-j, & k\le j,\\
        d-k, & j<k<d,\\
        0, & k=d,
        \end{cases}
\]
and define
\begin{equation}
\label{eq:S-P-linear-extension}
        S(w)=\sum_{i=0}^d i w_i,
        \qquad
        P_j(w)=-\sum_{k=0}^d \phi_j(k)w_k
        \quad (1\le j\le d-1).
\end{equation}
When $\sum_i w_i=0$, these definitions agree with \eqref{eq:S-P-def}.

The following lattice lemma is the deterministic reduction to the two
anti-diagonals.
\begin{lemma}
\label{lem:refined-lattice}
There is an absolute constant $C_L$ such that, for every $d\ge2$ and every
$z\in L_d$, there is a representation
\begin{equation}
\label{eq:refined-representation}
        z=
        \sum_{\substack{0\le a,b\le d-1\\a+b\in\{d-2,d-1\}}}
        \lambda_{a,b}\beta_{a,b},
        \qquad
        \lambda_{a,b}\in\mathbb Z,
\end{equation}
with
\begin{equation}
\label{eq:refined-cost}
        \sum_{a,b}|\lambda_{a,b}|
        \le
        C_L\left(
        \sum_{j=1}^{d-1}|P_j(z)|+d|S(z)|
        \right).
\end{equation}
\end{lemma}

\begin{proof}
In this proof we use the notation $\beta_{a,b}=f_a+f_b$ for ordered pairs as
well; of course $\beta_{a,b}=\beta_{b,a}$.  In the $f$-coordinate system, the
generator $\beta_{a,b}$ is the vector $\varepsilon_a+\varepsilon_b$ in
$\mathbb Z^d$.  Put
\[
        A_i=\varepsilon_i+\varepsilon_{d-1-i}
        \quad (0\le i\le d-1),
        \qquad
        B_i=\varepsilon_i+\varepsilon_{d-2-i}
        \quad (0\le i\le d-2).
\]
These are exactly the generators on the two anti-diagonals.  For
$0\le i\le d-2$,
\[
        A_i-B_i=\varepsilon_{d-1-i}-\varepsilon_{d-2-i}.
\]
Thus every adjacent difference $\varepsilon_j-\varepsilon_{j-1}$ is generated
with absolute coefficient cost at most $2$.

One of $d-1$ and $d-2$ is even.  Therefore one of the two anti-diagonals contains
a loop generator $2\varepsilon_m$.  Using adjacent differences,
\begin{equation*}
        2\varepsilon_0
        =
        2\varepsilon_m+
        2\sum_{r=1}^{m}(\varepsilon_{r-1}-\varepsilon_r)
\end{equation*}
is generated with coefficient cost $O(d)$.

By definition of $S(z)$ and $P_j(z)$,
\begin{equation}
\label{eq:c-decomposition}
        (c_0,\ldots,c_{d-1})
        =
        S(z)\varepsilon_0+
        \sum_{j=1}^{d-1}P_j(z)(\varepsilon_j-\varepsilon_{j-1}).
\end{equation}
Since $S(z)$ is even, the first term is an integer multiple of
$2\varepsilon_0$.  The previous two paragraphs and \eqref{eq:c-decomposition}
give \eqref{eq:refined-representation} and \eqref{eq:refined-cost}.
\end{proof}

\section{Random discrepancy in the refined norm}
\label{sec:random-discrepancy}

Let
\[
        M_i=m(H_0,i),
        \qquad
        M=(M_0,\ldots,M_d).
\]
For every vertex $v$, conditioning on $X_v=x$ gives
\[
        d_{H_0}(v)\sim \Bin(d,x).
\]
Therefore
\begin{equation}
\label{eq:uniform-degree-law}
        \Prob(d_{H_0}(v)=i)
        =
        \binom{d}{i}\int_0^1 x^i(1-x)^{d-i}\,dx
        =
        \frac{1}{d+1},
        \qquad
        0\le i\le d.
\end{equation}
Thus $\E M_i=n/(d+1)$.

The next lemma bounds the refined discrepancy of the threshold graph.
\begin{lemma}
\label{lem:expected-refined}
Let $q$ satisfy \eqref{eq:target-vector}, and put
\[
        z=q-M.
\]
There is an absolute constant $C_R$ such that
\begin{equation}
\label{eq:expected-refined}
        \E\left[
        \sum_{j=1}^{d-1}|P_j(z)|+d|S(z)|
        \right]
        \le
        C_R(d^2\sqrt n+d^3).
\end{equation}
\end{lemma}

\begin{proof}
Fix $1\le j\le d-1$.  By the linear extension in
\eqref{eq:S-P-linear-extension},
\[
        P_j(M)
        =
        -\sum_{v\in V(G)}\phi_j(d_{H_0}(v)).
\]
The function $\phi_j$ is $1$-Lipschitz on $\{0,\ldots,d\}$.  Fix a vertex
$u$ and change the single label $X_u$, keeping all other labels fixed.  Then
only the edges incident with $u$ can change.  Consequently $d_{H_0}(u)$ can
change by at most $d$, while for each neighbour $w\in N_G(u)$ the degree
$d_{H_0}(w)$ can change by at most $1$; all other degrees are unchanged.
Hence
\[
        |P_j(M')-P_j(M)|
        \le d+\sum_{w\in N_G(u)}1
        \le 2d .
\]
Thus, as a function of the independent labels $(X_v)_{v\in V(G)}$, the random
variable $P_j(M)$ has the bounded-differences property with constants at most
$2d$.  By the bounded-differences form of the Efron--Stein inequality
\cite[Corollary~3.2]{BLM},
\begin{equation*}
        \Var(P_j(M))
        \le
        \frac14\sum_{v\in V(G)}(2d)^2
        =
        nd^2.
\end{equation*}
Since \eqref{eq:target-vector} and \eqref{eq:uniform-degree-law} imply
$|q_i-\E M_i|\le 1$ for every $i$,
\begin{equation*}
        |P_j(q)-\E P_j(M)|\le C d^2.
\end{equation*}
Therefore
\begin{equation}
\label{eq:Pj-first-moment}
        \E|P_j(z)|
        \le
        C d\sqrt n+C d^2.
\end{equation}
Summing \eqref{eq:Pj-first-moment} over $j=1,\ldots,d-1$ gives
\begin{equation}
\label{eq:Pj-sum-bound}
        \E\sum_{j=1}^{d-1}|P_j(z)|
        \le
        C d^2\sqrt n+C d^3.
\end{equation}

It remains to estimate $S(z)$.  We have
\[
        S(M)=\sum_{i=0}^d iM_i=2|E(H_0)|.
\]
Changing one label toggles at most $d$ incident edges, so $S(M)$ changes by at most $2d$. Again by \cite[Corollary~3.2]{BLM}, 
\begin{equation*}
\Var(S(M)) \le \frac14\sum_{v\in V(G)}(2d)^2 = nd^2. 
\end{equation*}
Also,
\begin{equation*}
        |S(q)-\E S(M)|\le C d^2.
\end{equation*}
Thus
\begin{equation}
\label{eq:S-first-moment}
        \E|S(z)|\le C d\sqrt n+C d^2.
\end{equation}
Combining \eqref{eq:Pj-sum-bound} and \eqref{eq:S-first-moment} proves
\eqref{eq:expected-refined}.
\end{proof}

The following supply lemma gives, simultaneously for all signed anti-diagonal
types, enough candidate edges for the later greedy selection.
\begin{lemma}
\label{lem:supply}
Fix $\varepsilon>0$.  For every simple $d$-regular graph $G$ on $n$ vertices
with $2\le d\le n^{1/12-\varepsilon}$, with probability $1-o(1)$, every signed
type
\[
        (a,b,+),\ (a,b,-),
        \qquad
        0\le a\le b\le d-1,
        \qquad
        a+b\in\{d-2,d-1\},
\]
has at least
\begin{equation}
\label{eq:supply-lower}
        c_S n d^{-3}
\end{equation}
candidate edges, where $c_S>0$ is an absolute constant and the $o(1)$ is uniform
in $G$ and $d$ in this range.  Here $+$ denotes a positive candidate and $-$
denotes a negative candidate.
\end{lemma}

\begin{proof}
Let $Y_\tau$ be the number of candidates of a fixed signed type $\tau$.  We use
the standard dependency-graph second-moment bound, with the dependency graph
verified explicitly below.  By Lemma~\ref{lem:edge-atom},
\begin{equation*}
        \mu_\tau:=\E Y_\tau
        \ge
        c nd^{-3}.
\end{equation*}
Write $I_e$ for the indicator that an edge $e$ has signed type $\tau$.  For a
fixed edge $e=uv$, the event $\{I_e=1\}$ depends only on the labels in the
closed neighborhood $A_e=N[u]\cup N[v]$, where $N[u]=N_G(u)\cup\{u\}$ denotes the closed neighborhood.  If $A_e\cap A_f=\emptyset$, then
$I_e$ and $I_f$ are independent.  To count the possible exceptions, fix
$w\in A_e$.  If $w\in A_f$ for an edge $f=xy$, then
$w\in N[x]$ or $w\in N[y]$, and hence, by symmetry of the graph metric,
$x\in N[w]$ or $y\in N[w]$.  Thus at least one endpoint of $f$ lies in the
closed neighborhood $N[w]$, giving at most $d(d+1)$ possible edges $f$ for this
fixed $w$.  Since $|A_e|\le2d+2$, for each fixed $e$ there are at most
$2d(d+1)^2=O(d^3)$ edges $f$ for which independence is not forced.
Let $\mathcal D(e)$ denote this set of possibly dependent edges.  Then
\[
\begin{aligned}
        \Var(Y_\tau)
        &=\sum_e \Var(I_e)+2\sum_{e<f}\operatorname{Cov}(I_e,I_f) \\
        &\le \sum_e \Prob(I_e=1)
        +2\sum_e\sum_{f\in\mathcal D(e)}\Prob(I_e=1)
        \le C d^3\sum_e \Prob(I_e=1),
\end{aligned}
\]
since independent pairs have zero covariance and
$\operatorname{Cov}(I_e,I_f)\le \Prob(I_e=1)$.  Therefore
\begin{equation*}
        \Var(Y_\tau)\le C d^3\mu_\tau.
\end{equation*}
Chebyshev's inequality gives
\begin{equation*}
        \Prob\left(Y_\tau<\frac{\mu_\tau}{2}\right)
        \le
        \frac{C d^3}{\mu_\tau}
        \le
        \frac{C d^6}{n}.
\end{equation*}
There are $O(d)$ signed types, so the union bound gives failure probability
\[
        O\left(\frac{d^7}{n}\right)=o(1)
\]
under \eqref{eq:d-range}.  This proves the lemma.
\end{proof}

\section{Completion of the proof}
\label{sec:completion}

We now prove Theorem~\ref{thm:main}.
\begin{proof}
Fix $\varepsilon>0$, and let $n$ be sufficiently large.  Expose the threshold
random subgraph $H_0$, set $z=q-M$, and put
\[
        R_0=
        \sum_{j=1}^{d-1}|P_j(z)|+d|S(z)|.
\]
Let $\mathcal A$ be the event that all signed types have supply at least
$c_Snd^{-3}$.  By Lemma~\ref{lem:supply}, $\Prob(\mathcal A)=1-o(1)$.  By
Lemma~\ref{lem:expected-refined} and the inequality $d^3\le d^2\sqrt n$, valid
for all sufficiently large $n$ in the range \eqref{eq:d-range},
\[
        \E R_0\le C d^2\sqrt n.
\]
Markov's inequality gives
\[
        \Prob\left(R_0\le 4C d^2\sqrt n\right)\ge \frac34.
\]
For all sufficiently large $n$, also $\Prob(\mathcal A)>3/4$.  Hence the event
\[
        \mathcal A\cap\left\{R_0\le 4C d^2\sqrt n\right\}
\]
has positive probability.  Fix a realization of $H_0$ in this intersection.
After increasing $C$, this realization satisfies
\begin{equation*}
        \sum_{j=1}^{d-1}|P_j(z)|+d|S(z)|
        \le
        C d^2\sqrt n.
\end{equation*}

Both $q$ and $M$ have coordinate sum $n$.  Also
\[
        \sum_{i=0}^d iM_i=2|E(H_0)|
\]
is even, while $\sum_i iq_i$ is even by \eqref{eq:target-vector}.  Hence
$z=q-M\in L_d$.  Lemma~\ref{lem:refined-lattice} gives
\begin{equation}
\label{eq:z-representation}
        z=
        \sum_{\substack{0\le a,b\le d-1\\a+b\in\{d-2,d-1\}}}
        \lambda_{a,b}\beta_{a,b},
\end{equation}
with
\begin{equation}
\label{eq:R-bound-final}
        R:=\sum_{a,b}|\lambda_{a,b}|
        \le
        C d^2\sqrt n.
\end{equation}

The representation in \eqref{eq:z-representation} uses ordered pairs.  Since
$\beta_{a,b}=\beta_{b,a}$, combine the coefficients by setting, for $a<b$,
\[
        \Lambda_{a,b}=\lambda_{a,b}+\lambda_{b,a},
        \qquad
        \Lambda_{a,a}=\lambda_{a,a}.
\]
Then
\begin{equation*}
        z=
        \sum_{\substack{0\le a\le b\le d-1\\a+b\in\{d-2,d-1\}}}
        \Lambda_{a,b}\beta_{a,b},
        \qquad
        \sum_{a\le b}|\Lambda_{a,b}|\le R.
\end{equation*}

For a pair $(a,b)$ with $\Lambda_{a,b}>0$, we shall add
$\Lambda_{a,b}$ positive candidates of type $(a,b)$.  For a pair $(a,b)$ with
$\Lambda_{a,b}<0$, we shall delete $|\Lambda_{a,b}|$ negative candidates of type
$(a,b)$.

All candidate types in this paragraph are evaluated with respect to the original
threshold graph $H_0$; the selected edge additions and deletions are applied only
after all choices have been made.  We choose all these candidate edges greedily,
requiring them to be vertex-disjoint.  At any point fewer than $R$ edges have
already been selected.
A previously selected edge removes itself if it has the same signed type, and it
forbids at most $2d$ further candidates of any fixed signed type that share one
of its endpoints.  Thus fewer than $(2d+1)R$ candidates of any fixed signed type
are unavailable.  By \eqref{eq:R-bound-final},
\[
        (2d+1)R\le C d^3\sqrt n.
\]
On the other hand, by \eqref{eq:supply-lower}, the supply of each signed type is
at least $c_Snd^{-3}$.  Since $d\le n^{1/12-\varepsilon}$,
\[
        \frac{c_Snd^{-3}}{C d^3\sqrt n}
        \ge
        c'\frac{\sqrt n}{d^6}
        \ge
        c'n^{6\varepsilon}\to\infty.
\]
Thus, for all sufficiently large $n$, the supply of each signed type is larger
than the number of candidates of that signed type that can have been made
unavailable by the previously selected edges.  Therefore the greedy procedure can
select all required candidates vertex-disjointly.

Starting from $H_0$, add all selected positive candidates and delete all selected
negative candidates.  Since the selected edges are vertex-disjoint, the local
degree-count changes do not interfere.  The total change in the degree-count
vector is
\[
        \sum_{a\le b}\Lambda_{a,b}\beta_{a,b}=z.
\]
Hence the resulting spanning subgraph $H$ satisfies
\[
        (m(H,0),\ldots,m(H,d))=M+z=q.
\]
This proves \eqref{eq:exact-target}.

It remains to justify the final consequence of the theorem.  Write
\[
        n=(d+1)a+r,
        \qquad
        0\le r\le d.
\]
If $r>0$, choose a subset $S\subseteq\{0,1,\ldots,d\}$ of size $r$ such that
\[
        a\sum_{i=0}^d i+\sum_{i\in S}i
        \equiv 0\pmod 2.
\]
Both parities of $\sum_{i\in S}i$ are attainable.  Indeed,
start with any $r$-set.  Since $1\le r\le d$, both this set and its complement
are nonempty.  If there were no two elements of opposite parity on different
sides of the cut, then all elements of $\{0,1,\ldots,d\}$ would have the same
parity, impossible for $d\ge2$.  Thus swapping two elements of opposite parity
across the cut flips the parity of the sum while preserving the size $r$.  Define
$q_i=a+1$ for $i\in S$ and $q_i=a$ otherwise.

If $r=0$, start with $q_i=a$ for all $i$.  If $\sum_i iq_i$ is even, keep this
choice.  If it is odd, replace $q_0,q_1$ by $a+1,a-1$, respectively.  Uniformly
in the range \eqref{eq:d-range}, we have
\[
        a=\frac{n}{d+1}\to\infty,
\]
so $a-1\ge0$ for all sufficiently large $n$.  The new vector still satisfies
\eqref{eq:target-vector}.  Thus a parity-compatible target vector satisfying
\eqref{eq:target-vector} always exists.  Applying the exact statement to this
$q$ proves the last assertion.
\end{proof}

\section*{Acknowledgments and AI disclosure}
The third author thanks Shanghai Institute for Mathematics and Interdisciplinary Sciences (SIMIS), China for their financial support. This research was partly funded by SIMIS, China under grant number SIMIS-ID-2024-WE. The third author is grateful for the resources and facilities provided by SIMIS, which were essential for the completion of this work.

During the development and preparation of this work, the authors used ChatGPT for preliminary, non-authoritative assistance, including organizational discussion, language polishing, and exploratory discussion of possible proof strategies.  The tool was not treated as a mathematical authority or cited source.  No mathematical claim, computation, or argument was included in the paper without independent verification, substantial revision, and final approval by the authors.  All theorem statements, proofs, computations, references, and final arguments were checked, revised, and finalized by the authors, who take full responsibility for the correctness, originality, and integrity of the paper.

\end{document}